\documentclass[11pt,a4paper]{amsart}
\usepackage{amssymb, amstext, amscd, amsmath}
\usepackage{url}
\usepackage{amsfonts}
\usepackage{lscape}
\usepackage{amsthm}
\usepackage{amsfonts}
\usepackage{array}
\usepackage{eucal}
\usepackage{latexsym}
\usepackage{mathrsfs}
\usepackage{textcomp}
\usepackage{verbatim}
\usepackage{tikz}

\newtheorem{thm}{Theorem}[section]

\newtheorem{con}[thm]{Conjecture}

\newtheorem{lem}[thm]{Lemma}

\numberwithin{equation}{section}

\newcommand{\bQ}{{\mathbb{Q}}}
\newcommand{\R}{{\mathbb{R}}}

\newcommand{\bZ}{{\mathbb{Z}}}

\newcommand{\bP}{{\mathbb{P}}}
  \newcommand{\E}{{\mathcal{E}}}

  \newcommand{\M}{{\mathcal{M}}}

\renewcommand{\P}{{\mathcal{P}}}

  \newcommand{\V}{{\mathcal{V}}}

\newcommand{\rank}{\operatorname{rank}}

\newcommand{\inc}{\delta}
\newcommand{\sm}{\setminus}

\newcommand{\dl}{{[d-2]}}
\newcommand{\dlm}{{[d-1]}}
\newcommand{\dmt}{{[d-t]}}
\newcommand{\kl}{{[k]}}
\newcommand{\one}{{[1]}}
\tikzset{nodeblack/.style={circle,draw=black,fill=black!30,inner sep=1.2pt}}
\tikzset{every loop/.style={}}
\begin{document}

\title{Rigidity of linearly-constrained frameworks}

\author[James Cruickshank]{James Cruickshank}
\address{School of Mathematics,  Statistics and Applied Mathematics, National University of
Ireland, Galway, Ireland}
\email{james.cruickshank@nuigalway.ie}
\author[Hakan Guler]{Hakan Guler}
\address{School of Mathematical Sciences, Queen Mary
University of London, Mile End Road, London E1 4NS, U.K. }
\email{hakanguler19@gmail.com}
\author[Bill Jackson]{Bill Jackson}
\address{School of Mathematical Sciences, Queen Mary
University of London, Mile End Road, London E1 4NS, U.K. }
\email{b.jackson@qmul.ac.uk}
\author[Anthony Nixon]{Anthony Nixon}
\address{Department of Mathematics and Statistics\\ Lancaster University\\
LA1 4YF \\ U.K. }
\email{a.nixon@lancaster.ac.uk}
\date{\today}

\begin{abstract}
We consider the problem of characterising the generic rigidity  of bar-joint frameworks in $\R^d$ in which each vertex is constrained to lie in a given affine subspace. The special case when $d=2$ was previously solved by I. Streinu and L. Theran in 2010. We will extend their characterisation to the case when $d\geq 3$ and each vertex is constrained to lie in an affine subspace of dimension $t$, when $t=1,2$ and also when $t\geq 3$ and $d\geq t(t-1)$.
 We then point out that results on body-bar frameworks obtained by N. Katoh and S. Tanigawa in 2013 can be used to characterise when a graph has a rigid realisation as a $d$-dimensional body-bar framework with a given set of linear constraints.
\end{abstract}


\maketitle

\section{Introduction}\label{introduction}

A (bar-joint) framework $(G,p)$ in $\mathbb{R}^d$ is  a
finite graph $G=(V,E)$ together with a  realisation $p:V\rightarrow
\mathbb{R}^d$. The framework $(G,p)$ is \emph{rigid} if every edge-length
preserving continuous motion of the vertices arises as a congruence
of $\mathbb{R}^d$.

It is NP-hard to determine whether a given framework is rigid \cite{Abb}, but this problem becomes more tractable for generic frameworks. It is known that the rigidity of a generic framework $(G,p)$ in $\mathbb{R}^d$ depends only on the underlying graph $G$, see \cite{AR}. We say that  {\em $G$ is rigid in $\mathbb{R}^d$} if some/every generic realisation of $G$ in $\mathbb{R}^d$ is rigid.   Combinatorial characterisations of generic rigidity in $\mathbb{R}^d$ have been obtained when $d\leq 2$, see \cite{laman}, and these characterisations give rise to efficient combinatorial algorithms to decide if a given graph is rigid. In higher dimensions, however, no combinatorial characterisation or algorithm is yet known.

Motivated by numerous potential applications, notably in mechanical engineering, rigidity has also been considered for frameworks with various kinds of pinning constraints \cite{EJNSTW,KatTan,ShaiServWh,ST, Tetal}. Most relevant to this paper is
the work of Streinu and Theran  \cite{ST} on slider-pinning, which we describe below.

Throughout this paper we will use the term {\em graph} to describe a graph which may contain
multiple edges and loops and denote such a graph by $G=(V,E,L)$ where $V,E,L$ are the sets of vertices, edges and loops, respectively. We will use the terms {\em simple graph} to describe a graph which contains neither multiple edges nor loops, and {\em looped simple graph} to describe a graph which contains no multiple edges but may contain loops.

A {\em linearly-constrained framework in $\R^d$}  is a triple $(G,
p, q)$ where $G=(V,E,L)$ is a graph, $p:V\to \R^d$ and
$q:L\to \R^d$. For $v_i\in V$ and $e_j\in L$ we put $p(v_i)=p_i$ and
$q(e_j)=q_j$.
It is {\em generic} if $(p, q)$ is algebraically independent over
$\mathbb{Q}$.

An {\em infinitesimal motion} of $(G, p, q)$ is a map $\dot
p:V\to \R^d$ satisfying the system of linear equations:
\begin{eqnarray}
\label{eqn1} (p_i-p_j)\cdot (\dot p_i-\dot p_j)&=&0 \mbox{ for all $v_iv_j \in E$}\\
\label{eqn2} q_j\cdot \dot p_i&=&0 \mbox{ for all incident pairs $v_i\in V$ and
$e_j \in L$.}
\end{eqnarray}
The second constraint implies that  the infinitesimal velocity of each
$v_i\in V$ is constrained to lie on the hyperplane through $p_i$ with normal $q_j$
for each loop $e_j$ incident to $v_i$.

The {\em rigidity matrix $R (G, p, q)$} of the linearly-constrained framework $(G, p, q)$ is  the
matrix of coefficients of this system of equations for the unknowns
$\dot p$. Thus $R (G, p, q)$ is a $(|E|+|L|)\times d|V|$ matrix, in
which: the row indexed  by an edge $v_iv_j\in E$ has $p(u)-p(v)$ and
$p(v)-p(u)$ in the $d$ columns indexed by $v_i$ and $v_j$,
respectively and zeros elsewhere; the row indexed  by a loop
$e_j=v_iv_i\in L$ has $q_j$  in the $d$ columns indexed by $v_i$ and
zeros elsewhere. The $|E|\times d|V|$ sub-matrix consisting of the rows indexed by $E$ is the {\em bar-joint rigidity matrix} $R(G-L,p)$ of the bar-joint framework $(G-L,p)$.

The framework $(G, p, q)$ is {\em infinitesimally rigid} if its only
infinitesimal motion is $\dot p =0$, or equivalently if $\rank R(G,
p, q) = d|V|$. We say that the graph $G$ is  {\em rigid} in $\R^d$
if $\rank R(G, p, q) = d|V|$ for some realisation $(G,p, q)$ in
$\R^d$, or equivalently if $\rank R(G, p, q) = d|V|$ for all {\em
generic} realisations $(G,p,q)$ i.e.\  all realisations for which
$(p,q)$ is algebraically independent over $\bQ$.

Streinu and Theran \cite{ST} characterised the graphs
$G$ which are rigid in $\mathbb{R}^2$.
We need to introduce some terminology to describe their result. Let
$G=(V,E,L)$ be a graph. For $F\subseteq E\cup L$, let $V_F$ denote the set of vertices incident to $F$.

\begin{thm}\label{thm:ST}
A  graph can be realised as an infinitesimally rigid
linearly-constrained framework in $\R^2$ if and
only if it has a spanning subgraph $G=(V,E,L)$ such that
$|E|+|L|=2|V|$, $|F|\leq 2|V_F|$ for all $F\subseteq E\cup L$ and $|F|\leq 2|V_F|-3$ for all $\emptyset\neq F\subseteq E$.
\end{thm}

The results of this paper extend Theorem \ref{thm:ST} to $\R^d$
under the assumption that the number of linear constraints at each vertex is sufficiently large compared to $d$.
We need some further definitions to state our main results.
A graph $G=(V,E,L)$ is \emph{$(k,\ell)$-sparse} if $|F|\leq k|V_F|-\ell$
holds for all $\emptyset\neq F\subseteq E\cup L$. The graph $G$
 is \emph{$(k,\ell)$-tight} if it is  $(k,\ell)$-sparse and $|E|+|L|=k|V|-\ell$.
We use $G^\kl$ to denote the graph obtained from $G$ by adding $k$ loops to each vertex of $G$.

\begin{thm}\label{thm:main}
Suppose $G$ is a  graph and $d,t$ are positive integers with $d\geq \max\{2t,t(t-1)\}$. Then $G^\dmt$ can be realised as an infinitesimally rigid
linearly-constrained framework in $\R^d$ if and
only if $G$ has a $(t,0)$-tight looped simple spanning  subgraph.
\end{thm}

The complete simple graph $K_{2t+1}$ shows that the condition $d\geq 2t$ in Theorem \ref{thm:main} is, in some sense, best possible: $K_{2t+1}$ is $(t,0)$-tight and $K_{2t+1}^\dmt$ does not have an infinitesimally rigid realisation in $\R^d$ when $d=2t-1$. This follows from the fact that the rows of the bar-joint rigidity matrix of any realisation of $K_{2t+1}$ in $\R^{2t-1}$ are dependent. Hence the rank of the rigidity matrix of any realisation of $K_{2t+1}^{[t-1]}$ in $\R^{2t-1}$ will be at most $|E(K_{2t+1}^{[t-1]})|-1=(2t-1)|V(K_{2t+1}^{[t-1]})|-1$. We do not know whether the conclusion of Theorem \ref{thm:main}  still holds if the condition $d\geq t(t-1)$ is removed - it is conceivable that this condition is an artifact of our proof technique.

Linearly-constrained frameworks naturally arise when considering the infinitesimal rigidity of a bar-joint framework $(G,p)$ in $\R^d$ under the additional constraint that the vertices of $G$  lie on a smooth algebraic variety $\V$. We can model this situation as a linearly-constrained framework $(G^\dmt,p,q)$ in which $t=\dim \V$ and $q$ is chosen so that the image of the loops of $G^\dmt$ at each vertex $v$ span the orthogonal complement of the tangent space of $\V$ at $p(v)$.
 In this context, the continuous isometries
of $\V$ will always induce infinitesimal motions of $(G,p)$. We say
that {\em $(G,p)$ is infinitesimally rigid on $\V$} if these are the
only infinitesimal motions of $(G,p)$. Equivalently $(G,p)$ is
infinitesimally rigid on $\V$ if $\rank R(G^\dmt, p, q) = d|V|-\alpha$
where $\alpha$ is the {\em type} of $\V$ i.e.\  the dimension of the space
of infinitesimal isometries of $\V$.
The special case when $\V$ is an irreducible surface
in $\R^3$ was previously studied
by Nixon,
Owen and Power \cite{NOPa,NOP}.
They characterised generic rigidity for all such surfaces of types 3, 2 and 1.

\begin{thm}\label{thm:type1}
Let $G=(V,E)$ be a simple graph and let $\M$  be an irreducible
surface in $\R^3$ of type $\alpha\in \{1,2,3\}$. Then a generic framework
$(G,p)$ on $\M$ is infinitesimally rigid on $\M$ if and only if $G$
has a $(2,\alpha)$-tight spanning subgraph.
\end{thm}

Theorem \ref{thm:main} characterises the graphs $G$ which can be realised as an infinitesimally rigid linearly constrained framework on some type 0, $t$-dimensional variety
in $\R^d$, whenever $d\geq \max\{2t,t(t-1)\}$. The special case when $t=d-2$ and $d\geq 4$ characterises the graphs which have  an infinitesimally rigid realisation as a linearly-constrained framework on some type 0 surface in $\R^d$. A characterisation for the case when $d=2$ is given by Theorem \ref{thm:ST}. Our next result gives a solution to the remaining case when $d=3$.

\begin{thm}\label{thm:plane}
Suppose $G$ is a graph. Then $G^\one$ can be realised as an infinitesimally rigid
linearly-constrained framework in $\R^3$ if and
only if $G$ has a $(2,0)$-tight looped simple spanning  subgraph which contains no copy of $K_5$.
\end{thm}

Theorem \ref{thm:plane} completes the characterisation of infinitesimal rigidity of generic {\em plane-constrained frameworks} i.e.\  frameworks in which each vertex is constrained to lie on a given plane in $\R^d$. In this context it is natural to consider the infinitesimal rigidity of {\em line-constrained frameworks} in $\R^d$. Given a graph  $G$ and an integer $d\geq 2$, Theorem \ref{thm:main} tells us that $G^\dlm$ can be realised as an infinitesimally rigid
linearly-constrained framework in $\R^d$ if and
only if $G$ has a $(1,0)$-tight looped simple spanning  subgraph.
We will prove a more general result. We consider the case when we are given a  graph $G$ and a map $q$ from the set of loops of $G^\dlm$ to $\R^d$ such that the image of the set of loops at each vertex of $G^\dlm$ spans a subspace of dimension at least $d-1$. We then characterise when there exists a map $p:V\to \R^d$ such that $(G,p,q)$ is infinitesimally rigid.

Our final results concern  {\em body-bar frameworks in $\R^d$} i.e.\  frameworks consisting of $d$-dimensional rigid bodies joined by rigid bars. Tay \cite{T} characterised when a graph has an infinitesimally rigid  realisation as a body-bar framework in $\R^d$. We point out that two results of Katoh and Tanigawa \cite{KatTan} immediately extend Tay's result to linearly-constrained body-bar frameworks in $\R^d$.

\section{Linearly-constrained frameworks}
 We will prove Theorem \ref{thm:main}. The necessity part of the theorem follows immediately from:

 \begin{lem}\label{lem:nec}
Suppose $G$ is a  graph. If $G^\dmt$ can be realised as an infinitesimally rigid
linearly-constrained framework in $\R^d$ then $G$ has a $(t,0)$-tight looped simple spanning  subgraph.
\end{lem}
 \begin{proof}
 We may suppose that $(G^{[d-t]},p,q)$ is an infinitesimally rigid  generic realisation of $G^{[d-t]}$ in $\R^d$. Let $S$ be a set of loops of $G^{[d-t]}$ consisting of exactly $d-t$ loops at each vertex. It is not difficult to see that the rows of $R(G^{[d-t]},p,q)$ labeled by $S$ are linearly independent and hence we can choose a spanning subgraph $H=(V,E_H,L_H)$ of $G$ such that the rows of $R(H^{[d-t]},p,q|_{H^{[d-t]}})$ are linearly independent and $\rank R(H^{[d-t]},p,q|_{H^{[d-t]}})=d|V|$. The linear independence of the rows of $R(H^{[d-t]},p,q|_{H^{[d-t]}})$ immediately implies that $H$ has no multiple edges.
 If $H$ had a subgraph $F=(V_F,E_F,L_F)$ with $|E_F|+|L_F|>t|V_F|$ then we would have $\rank R(F^{[d-t]},p,q|_{F^{[d-t]}})\leq d|V_F|<|E_{F^{[d-t]}}|+|L_{F^{[d-t]}}|$. This would contradict the fact that the rows of $R(H^{[d-t]},p,q|_{H^{[d-t]}})$ are linearly independent. Hence $H$ is $(t,0)$-tight.
\end{proof}

 Our proof of sufficiency is inductive and is based on a reduction lemma which reduces a $(t,0)$-tight graph to a smaller such graph and an extension lemma which extends a graph which has a generically infinitesimally rigid realisation in $\R^d$ to a larger such graph.

\subsection*{Reduction operation}
 We shall need the following well known result on tight subgraphs of a sparse graph $G=(V,E,L)$. It follows easily from the fact that the function $F\mapsto |V_F|$ for $F\subseteq E\cup L$ is submodular.

 \begin{lem}
    \label{lem:unions}
    Suppose that \( H,K \) are \( (t,0) \)-tight subgraphs of
    a \( (t,0) \)-sparse graph.
    Then \( H \cup K \) and $H\cap K$ are also \( (t,0) \)-tight.
    \qed
\end{lem}

 Given a vertex $v$ in a graph $G$, we use $N_G(v)$ to denote the {\em neighbour set} of $v$ (i.e.\  the set of vertices of $G-v$ which are adjacent to $v$), $L_G(v)$ to denote the set of loops incident to $v$,  and $\inc_G(v)$ to denote the number of edges and loops of $G$ which are incident to $v$, counting each loop once. We will suppress the subscript `$G$' when it is obvious which graph we are referring to.


\begin{lem}
    \label{lem:reduction}
 Suppose $t$ is a positive integer and $v$ is a vertex of a $(t,0)$-tight graph $G$. Let $k=\inc(v)-t$.   Then there are distinct vertices \( v_1,\ldots,v_k \in N(v) \)
    such that \( (G - v) \cup \{l_1,\ldots,l_k\} \) is
    \( (t,0) \)-tight, where \( l_i \) is a new loop added at \( v_i \) for all $1\leq i\leq k$.
\end{lem}

\begin{proof}
We show that there are distinct vertices \( v_1,\ldots,v_i \in N(v) \)
    such that \(H_i= (G - v) \cup \{l_1,\ldots,l_i\} \) is
    \( (t,0) \)-sparse for all $0\leq i\leq k$ by induction on $i$.
If \( i = 0 \) then \( G-v \) is \( (t,0) \)-sparse since it is a subgraph of $G$.


    Suppose inductively  that, for some    \(0\leq r \leq k-1\),
    we have vertices \( v_1,\ldots,
    v_r \in N(v) \) such that \(H_r
 \) is
    \( (t,0) \)-sparse. For a contradiction, suppose that there is
    no neighbour \( v_{r+1} \in N(v) \) such that
    \( H_{r+1}\) is also \( (t,0) \)-sparse.
    Then for every \( u \in N(v) -\{v_1,\ldots,v_r\} \)
    there is some \( (t,0) \)-tight subgraph of \( H_r\)
    that contains \( u \). We can now use Lemma \ref{lem:unions} to obtain
    a \( (t,0) \)-tight subgraph \( H \) of \(H_r \) that
    contains \(N(v) -\{v_1,\ldots,v_r\} \). It is
    possible that \( H \) contains some of \( l_1,\ldots,l_r \).
    Without loss of generality, suppose that \( l_1,\ldots,l_b \in H\)
    and \( l_{b+1},\ldots,l_r \not\in H \).  The subgraph of $G$ induced by
    $V(H)\cup \{v\}$ is obtained from $H$ by deleting the loops $l_1,\ldots,l_b$, and adding the vertex $v$ and $\inc(v)-(r-b)$ edges incident to $v$. Since $H$ is $(t,0)$-tight and
    $$\inc(v)-(r-b)-b=k+t-r\geq k+t-(k-1)>t,$$
    this contradicts the fact that $G$ is $(t,0)$-sparse.

    Thus there are vertices \( v_1,\ldots,v_k \in N(v) \)
    such that \( H_k \) is \( (t,0) \)-sparse.
    Finally observe that \( H_k \) has
    \( \inc(v) - k = t\) fewer edges than \( G \) and one less
    vertex and so is, in fact, \( (t,0) \)-tight.
\end{proof}

\subsection*{Extension operation}
\label{sec:extmoves}

Let $H=(V,E,L)$ be a graph and $d\geq 1,0\leq k\leq d$ be integers.
The \emph{$d$-dimensional $k$-loop extension operation} forms a new graph $G$
from $H$ by deleting $k$ loops incident to distinct vertices of $H$ and adding a new vertex
$v$ and $d+k$ new edges and loops incident to $v$,
with the proviso that at least $k$ loops are added at $v$ and exactly one new edge is added from $v$ to each of the end-vertices of the $k$ deleted loops. Note that, since $k\leq d$, these conditions imply that $v$ is incident to at most $d$ loops and at most $d$ edges, see Figure \ref{fig:1-ext}.

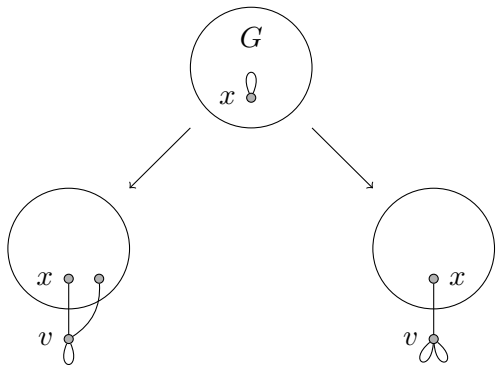
\begin{figure}[h]
\begin{center}
\begin{tikzpicture}[scale=0.8]
\node[] at (0,2) (G) [label=above:$G$]{};
\draw[] (0,2) circle (1cm);
\node[nodeblack] at (0,1.5) (x) [label=left:$x$]{}
 edge[in=113,out=67,loop]();

\draw[] (-3,-1) circle (1cm);
\node[nodeblack] at (-2.5,-1.5) (y2) []{};
\node[nodeblack] at (-3,-1.5) (x2) [label=left:$x$]{};
\node[nodeblack] at (-3,-2.5) (v2) [label=left:$v$]{}
 edge[] (x2)
 edge[bend right] (y2)
 edge[in=247,out=293,loop]();
\draw[->] (-1,1)--(-2,0);

\draw[] (3,-1) circle (1cm);
\node[nodeblack] at (3,-1.5) (x3) [label=right:$x$]{};
\node[nodeblack] at (3,-2.5) (v3) [label=left:$v$]{}
 edge[in=220,out=265,loop]()
 edge[in=275,out=320,loop]()
 edge[](x3);
\draw[->] (1,1)--(2,0);

\end{tikzpicture}
\end{center}
\caption{Possible $2$-dimensional $1$-loop extensions of a graph $G$.}
\label{fig:1-ext}
\end{figure}

\begin{lem}
    \label{lem:extension}
    Suppose that $G$ is obtained from $H$ by a $d$-dimensional $k$-loop extension operation which deletes a loop $l_j$ at $k$ distinct vertices $v_j$, $1\leq j\leq k$, of $H$ and adds a new vertex $v$. Suppose further that $H$ has an infinitesimally rigid realisation as a linearly-constrained framework in $\R^d$, and that $v_j$ is incident with at least \(\left\lceil \frac{(k-1)d}{k} \right \rceil
    \) loops in $H$ for all $1\leq j \leq k$ when $k\geq 2$. Then
    \( G \) has an infinitesimally rigid realisation as a linearly-constrained framework in \( \mathbb R^d \).
\end{lem}

\begin{proof}
    Suppose that \( (H,p_H,q_H) \) is generic and rigid where
    \( p_H:V_H \rightarrow
    \mathbb R^d\) and \( q_H:L_H \rightarrow \mathbb R^d \).
    For each vertex \( u \in V_H \), define \( p_G(u) = p_H(u) \)
    and for any loop \( l\in L_H \) at a vertex other than \( v_1,\ldots,v_k \),
    define \( q_G(l) = q_H(l) \).

    For \( 1\leq j \leq k \),
    let \( W_j \) be the linear span of
    \( q_H(L_H(v_j)) \), $A_j=p_H(v_j)+W_j$, and $A$ be the affine span of $p_H(N_G(v))$.
    Since \( (p_H,q_H) \) is generic, the affine spaces $A_j$, $1\leq j\leq k$, and $A$ are in general position\footnote{A set $\{B_1,B_2,\ldots,B_s\}$ of affine subspaces of $\R^d$ is in {\em general position} in $\R^d$ if for any two disjoint subsets $S,T$ of $\{1,2,\ldots,s\}$
such that $\bigcap_{i \in S} B_i \neq \emptyset$ and  $\bigcap_{i \in T} B_i \neq \emptyset$, we have
$\dim \bigcap_{i \in S \cup T} B_i = \dim \bigcap_{i \in S} B_i + \dim \bigcap_{i \in T} B_i  -d$  when $\dim \bigcap_{i \in S} B_i + \dim \bigcap_{i \in T} B_i \geq d$,  and $\bigcap_{i \in S \cup T} B_i = \emptyset$ when $\dim \bigcap_{i \in S} B_i + \dim \bigcap_{i \in T} B_i \leq d-1$.
},
    \(\dim(A_j) \geq  \left\lceil \frac{(k-1)d}k \right\rceil \) and $\dim A=|N_G(v)|-1\leq d-1$.
    An elementary dimension counting argument now implies that $(\bigcap_{j=1}^k A_j)\sm A\neq \emptyset$.
    Choose \( z \in (\bigcap_{j=1}^k A_j)\sm A\) and put $p_G(v)=z$.

We next define \( q_G \) on \( L_G(v),L_G(v_1),\ldots,L_G(v_k) \).
    We first choose $k$ loops $m_j\in L_G(v)$, $1\leq j\leq k$, and put \( q_G(m_j) = z-p_G(v_j)  =
    z - p_H(v_j)\). By the construction
    of \( z \) and since \( p_H \) is generic, we see that \( I=\{z-p_H(u)\,:\,u\in N_G(v)\}\) is a linearly independent set, and hence we may extend the definition of $q_G$ to the whole of
    $L_G(v)$ in such a way that $I\cup q_G(L_G(v)\sm \{m_1,m_2,\ldots,m_k\})$ is a basis for $\R^d$.

    Now, observe that for \( 1 \leq j \leq k \), \( q_G(m_j) \) is a
    nonzero element of \( W_j \). Thus we may define  \( q_G:
    L_G(v_j)\rightarrow \mathbb R^d\) so that \( q_G(L_G(v_j)) \cup q_G(m_j)\) is a spanning set of \( W_j \). (Note
    that \( L_G(v_j) = L_H(v_j) \sm \{l_j\} \)).

    We will show  that \( (G,p_G,q_G) \) is a rigid realisation of \( G \)
    in \( \mathbb R^d \). Suppose that \( \dot p: V_G \rightarrow
    \mathbb R^d\) is an infinitesimal motion of \( (G,p_G,q_G) \).
    Then,
    \(
        (\dot p(v)-\dot p(v_j))\cdot(p_G(v)-p_G(v_j)) = 0= \dot p(v)\cdot q_G(m_j)
    \) for all $1\leq j\leq k$.
    Since  \( q_G(m_j) = z - p_H(v_j)=p_G(v)-p_G(v_j)  \), this gives
    \( \dot p(v_j)\cdot q_G(m_j) = 0 \) for all $1\leq j\leq k$.
    Since we also have $\dot p(v_j)\cdot q_G(l)=0$ for all $l\in L_G(v_j)$ and \( q_G(L_G(v_j)) \cup q_G(m_j)\) is a spanning set of \( W_j \), \( \dot p(v_j) \in W_j^\perp \) for all $1\leq j\leq k$. This implies that
    \( \dot p \) restricts to an infinitesimal
    motion of \( (H,p_H,q_H) \). Thus \( \dot p(u) = 0 \) for all
    \( u \in V_H \).

    Finally, observe that since
    $I\cup q_G(L_G(v)\sm \{m_1,m_2,\ldots,m_k\})$ is a basis for $\R^d$
and \( \dot p(v) \) is orthogonal
    to all vectors in this basis, it must also be zero.
\end{proof}

\subsection*{Proof of Theorem \ref{thm:main}}
Necessity follows immediately from Lemma \ref{lem:nec},
so it only remains to prove sufficiency. We may assume without loss of generality that $G=(V,E,L)$ is a $(t,0)$-tight looped simple graph.
We will show that $G^\dmt$ has an infinitesimally rigid realisation in $\R^d$ by induction on $|V|$.
If $|V|=1$ then $G$ has $t$ loops. Hence $G^\dmt$ has $d$ loops and is generically rigid in
    \( \mathbb R^d \).

    Now suppose that \( |V|   \geq 2 \) and that
    the theorem is true for all graphs with at most \( |V|-1 \) vertices.
Let $\inc^*(u)$ denote the number of loops incident with each vertex $u$ of $G$.
Since \( |E| + |L| = t|V| \) and \( \sum_{u \in V} (\inc(u)+\inc^*(u)) = 2(|E|+|L|) \)
    there is  a \( v \in V \) such that \( \inc(v)+\inc^*(v) \leq 2t \). Also, since
    \( G \) is \( (t,0) \)-tight, \( \inc(v) \geq t \).
    Let \( k = \inc(v)-t  \).

    By Lemma \ref{lem:reduction} there are distinct vertices
    \( v_1,\ldots,v_k \in N(v)\)
    such that the graph $H$ obtained from $G - v$ by adding a new loop \( l_i \)  at \( v_i \)
    for all $1\leq i\leq k$
     is \( (t,0) \)-tight. By induction \( H^{[d-t]} \) is generically rigid in \(
    \mathbb R^d\).

     We will show that $G^\dmt$ is a $d$-dimensional $k$-loop extension of  $H^\dmt$.
This will follow from the hypothesis that $d\geq 2t$ and the facts that $k=\inc(v)-t$ and $t\leq \inc(v)\leq 2t$. These imply that $0\leq k\leq t$ and hence $k\leq d$ and $\delta_{G^\dmt}^*(v)\geq d-t\geq t\geq k$. In addition, for all $1\leq j\leq k$, we have
$$\delta_{H^\dmt}^*(v_j)\geq  d-t+1=\frac{(t-1)d}{t}+\frac{d}{t}-t+1 \geq \left\lceil \frac{(t-1)d}t \right\rceil
    \geq \left\lceil \frac{(k-1)d}k \right\rceil,$$
    since $d\geq t(t-1)$ and \( t \geq k \).
We can now use Lemma \ref{lem:extension} to deduce that $G^\dmt$ has an infinitesimally rigid realisation in $\R^d$.
\qed

\section{Plane-constrained frameworks}
\label{sec:planes}

We will prove Theorem \ref{thm:plane}. We define the {\em degree}, $\deg_G(x)$, of a vertex $v$ in a graph $G$ to be the number of edges and loops incident to $v$, counting each loop {\em twice}.  We say that a graph $G$ is {\em $k$-regular} if each vertex of $G$ has degree $k$.
Our proof technique of Theorem \ref{thm:plane} uses induction on the order of $G$ when $G$ is not simple and $4$-regular, combined with an ad hoc argument for this exceptional case.
We will need some further results to deal with this case.

\begin{lem}\label{lem:4reg}
Let $G=(V,E)$ be a 4-regular simple graph. Then $G$ is $(2,0)$-tight. Moreover if $G$ is connected then $|F|\leq 2|V_F|-1$ for all $F\subsetneq E$.
\end{lem}

\begin{proof}
Since $G$ is 4-regular we have $|E|=2|V|$. Suppose $G$ is not $(2,0)$-sparse. Then there is some $F\subset E$ with $|F|>2|V_F|$. This implies that $G[F]$ has average degree strictly greater than 4, contradicting the fact that $G$ is 4-regular.

Now assume $G$ is connected. Suppose $|F|=2|V_F|$ for some $F\subsetneq E$. This implies $G[F]$ has average degree exactly 4. Since $G$ is connected and $F\subsetneq E$ there exists a vertex $x\in V_F$ with $\deg_G(x)>4$, contradicting the fact that $G$ is 4-regular.
\end{proof}

The next result gives a sufficient condition for the rigidity matrix of a generic bar-joint framework
in $\R^3$ to have independent rows.

\begin{thm}(\cite[Theorem 3.5]{JJ})\label{thm:maxmin}
Let $G=(V,E)$ be a connected  simple graph on at least three
vertices with minimum degree at most $4$ and maximum degree at most
$5$, and $(G,p)$ be a generic realisation of $G$ as a bar-joint framework in $\R^3$. Then the
rows of the bar-joint rigidity matrix $R(G,p)$ are linearly independent if and only if $|F|\leq 3|V_F|-6$ for all $F\subseteq V$ with $|F|\geq 2$.
\end{thm}

Our next result concerns frameworks on surfaces.
Suppose $\M$ is an irreducible surface in $\mathbb{R}^3$
defined by a polynomial $f(x,y,z)=r$ and
$q=(x_1,y_1,z_1,\ldots,x_n,y_n,z_n)\in \R^{3n}$. Then the {\em family
of `concentric' surfaces induced by $q$}, $\M^q$, is the family
defined by the polynomials $f(x,y,z)=r_i$ where $r_i=f(x_i,y_i,z_i)$
for $1\leq i\leq n$.

\begin{lem}(\cite[Lemma 9]{JNstress})\label{lem:perturb}
Suppose  $(G,p)$ is an infinitesimally rigid framework on some
surface $\M$ in $\R^3$. Then $(G,q)$ is infinitesimally rigid on
$\M^q$ for all generic $q\in \R^{3|V|}$.
\end{lem}

We need two additional concepts for our next lemma. First, recall
that the edge sets of the simple $(2,1)$-sparse subgraphs of a graph
$G=(V,E)$ are the independent sets of a matroid on $E$. We call this
the \emph{simple $(2,1)$-sparse matroid} for $G$. Secondly, we define
an \emph{equilibrium stress} for a linearly-constrained  framework
$(G,p,q)$ in $\mathbb{R}^3$ to be a pair $(\omega,\lambda)$, where
$\omega:E\to \R$, $\lambda:L\to \R$ and $(\omega,\lambda)$ belongs
to the cokernel of $R(G,p,q)$.

\begin{lem}\label{lem:4regind}
Let $G=(V,E)$ a  $4$-regular connected  simple graph which is
distinct from $K_5$. Then $G^\one$ can be realised as an
infinitesimally rigid plane-constrained framework in
$\mathbb{R}^3$.
\end{lem}

\begin{proof}
Let $\E$ be the surface in
$\R^3$ defined by the equation $x^2+2y^2=1$. Then $\E$ is an elliptical
cylinder centred on the $z$-axis and has type 1. Let $p:V\to \R^3$
be generic, and $(G,p)$ be the corresponding framework on
the family of
concentric elliptical cylinders $\E^p$ induced by $p$. Lemma
\ref{lem:4reg} implies that $E$ is a circuit in the simple
$(2,1)$-sparse matroid for $G$. Theorem \ref{thm:type1} and Lemma
\ref{lem:perturb} now imply that $(G-e,p)$ is infinitesimally rigid on $\E^p$
for all $e\in E$. Hence the
only infinitesimal motions of $(G-e,p)$ on $\E^p$ are translations in
the direction of the $z$-axis.

Let $(G^{[1]},p,q)$ be the plane-constrained framework corresponding to $(G,p)$ on
$\E^p$. Then $(G^{[1]},p,q)$ has the same (1-dimensional) space of infinitesimal motions as $(G,p)$ on $\E^p$ and hence
$\rank R(G^{[1]},p,q)=\rank R(G^{[1]}-e,p,q)=3|V|-1$ for all $e\in E$.
This implies that $(G^{[1]},p,q)$ has a unique non-zero equilibrium stress $(\omega,\lambda)$ up to scalar multiplication.
Since $G$ is simple, 4-regular and distinct from $K_5$, we have
 $|F|\leq 3|V_F|-6$ for all $F\subset E$ with $|F|\geq 2$.
 Theorem \ref{thm:maxmin} now implies that the rows of $R(G^{[1]},p,q)$ indexed by $E$ are linearly independent
 and hence we must have
 $\lambda_f\neq 0$ for some $f\in L$. It follows that the
 matrix $R_f$ obtained from $R(G^{[1]},p,q)$ by deleting the row indexed by
 $f$ has $\ker R_f=\ker R(G^{[1]},p,q)$ and hence each $\dot p\in \ker R_f$ corresponds to a translation along the
 $z$-axis.
 Let $(G^{[1]},p,\tilde q)$ be the plane-constrained framework with $\tilde q(e)=q(e)$ for all $f\in L-f$ and
 $\tilde q(f)=(0,0,1)$. Then $\ker R(G,p,\tilde q)\subseteq \ker R_f$.
 The choice of $\tilde q(f)$ implies that no nontrivial translation along the $z$-axis can belong to
 $\ker R(G^{[1]},p,\tilde q)$. Hence $\ker R(G^{[1]},p,\tilde q)=\{0\}$ and $(G^{[1]},p,\tilde q)$ is an
 infinitesimally rigid plane-constrained framework in $\R^3$.
\end{proof}

We can now prove Theorem \ref{thm:plane}.


\subsection*{Proof of Theorem \ref{thm:plane}} We first prove necessity.
Suppose that $G^{\one}$ can be realised as an infinitesimally rigid plane-constrained framework in $\R^3$. We can show as in the proof of Lemma \ref{lem:nec} that $G$ has a spanning looped simple $(2,0)$-tight subgraph $H$, with a realisation $(H,p,q)$ such that
the rows of $R(H^\one,p,q)$ are linearly independent and $\rank R(H^\one,p,q)=d|V|$. If $H$ had a subgraph $K$ which is isomorphic to $K_5$, then the fact that $K_5$ is generically dependent as a bar-joint framework in $\R^3$ would imply that the rows of $R(H^{[1]},p,q)$ labelled by $E(K)$ are linearly dependent. Hence $H$ contains no copy of $K_5$.

We next prove sufficiency.  Suppose $G$ has a  spanning looped simple subgraph which is $(2,0)$-tight, and
  in addition, contains no copy of $K_5$. We will prove that $G^\one$ can be realised as an infinitesimally rigid
  plane-constrained framework in $\R^3$ by induction on $|V|+|E|+|L|$. We may
assume that $G$ is connected and  so  $|E|+|L|=2|V|$. If $G$ is the graph
with one vertex and two loops, then it is easy to see that $G^\one$
has an infinitesimally rigid realisation in $\R^3$, so we may assume that $|V|\geq 2$.
Let $ v $ be a vertex of $ G $ chosen so that $\delta (v) $ is as small as possible. Since $ G $ is $(2, 0) $-tight, $2\leq\delta (v)\leq 4 $.

Suppose $\delta (v)\in\{2, 3\} $. By Lemma \ref {lem:reduction}, $ G $ can be reduced to  smaller $(2,0)$-tight graph
$H$ by deleting $ v $ and adding $\delta(v)-2 $ loops to $ N (v) $.
By induction
$H^\one$ has an infinitesimally rigid realisation  in $\R^3$. We can now
apply Lemma \ref{lem:extension}  to deduce that
$G^\one$ has an infinitesimally rigid realisation.

 It remains to consider the case when $\delta (v)=4 $. Since $G$ is $(2, 0) $-tight, it must be 4-regular and simple. Since $ G $ is connected and not equal to $ K_5 $,  Lemma \ref {lem:4regind} now implies that $ G $ has an infinitesimally rigid realisation in $\R^3 $.
\qed

\section{Line-constrained frameworks}
\label{sec:lines}

 Given a graph $G$ and an integer $d\geq 2$, Theorem \ref{thm:main} implies that  $G^{[d-1]}$ can be realised as an infinitesimally rigid line-constrained framework in $\mathbb{R}^d$ if and only if $G$ has a spanning $(1,0)$-tight subgraph. We will prove
a more general result which allows nongeneric line constraints. We will need a lemma concerning non-generic $d$-dimensional $0$-loop extensions, which we will refer to simply as {\em $d$-dimensional $0$-extensions}. See Figure \ref{fig:0-ext}
for an illustration when $d=2$.

\begin{figure}[h]
\begin{center}
\begin{tikzpicture}[scale=0.8]
\node[] at (0,2) (G) [label=above:$G$]{};
\draw[] (0,2) circle (1cm);

\draw[] (-3,-1) circle (1cm);
\node[nodeblack] at (-3.5,-1.5) (x1) []{};
\node[nodeblack] at (-2.5,-1.5) (y1) []{};
\node[nodeblack] at (-3,-2.5) (v1) [label=left:$v$]{}
 edge[] (x1)
 edge[] (y1);
\draw[->] (-1,1)--(-2,0);

\draw[] (0,-1) circle (1cm);
\node[nodeblack] at (0,-1.5) (x2) []{};
\node[nodeblack] at (0,-2.5) (v2) [label=left:$v$]{}
 edge[] (x2)
 edge[in=247,out=293,loop]();
\draw[->] (0,0.7)--(0,0.2);

\draw[] (3,-1) circle (1cm);
\node[nodeblack] at (3,-2.5) (v3) [label=left:$v$]{}
 edge[in=220,out=265,loop]()
 edge[in=275,out=320,loop]();
\draw[->] (1,1)--(2,0);
\end{tikzpicture}
\end{center}
\caption{Possible $2$-dimensional $0$-extensions of a graph $G$.}
\label{fig:0-ext}
\end{figure}
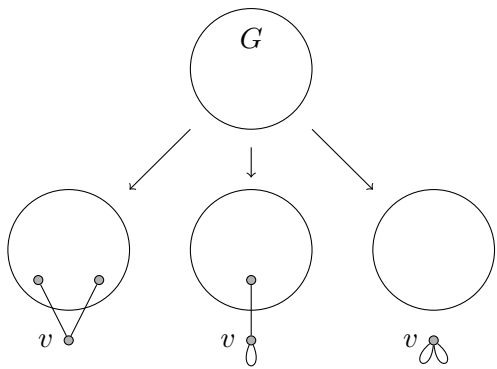

\begin{lem}\label{lem:02ext}
Let $G$ be a  graph and $H$ be constructed from $G$ by a $d$-dimensional $0$-extension operation which adds a new vertex $v_0$, new edges $v_0v_1,v_0v_2,\ldots, v_0v_t$, and new loops $e_{t+1},e_{t+2},\ldots,e_d$.
Suppose $(G,p,q)$ is a realisation of
$G$ in $\R^d$
and $(H,\hat p,\hat q)$ is a realisation of $H$ with $\hat p|_G=p$ and $\hat
q|_G=q$.
Then
$(H,\hat p,\hat q)$ is  infinitesimally rigid if and only if
$(G, p, q)$ is  infinitesimally rigid and $\{\hat p_0-\hat p_1,\hat p_0-\hat p_2,\ldots,\hat p_0-\hat p_t,\hat q(e_{t+1}),\ldots,\hat q(e_d)\}$ is
linearly independent.
\end{lem}

\begin{proof}
The rigidity matrix for  $(H,\hat p,\hat q)$ has the form
$$
R(H,\hat p,\hat q)=
\begin{pmatrix}
A&*\\
0&R(G,p,q)
\end{pmatrix}
$$
where
\[
A=\begin{pmatrix}
\hat p_0-\hat p_1\\
\vdots\\
\hat p_0-\hat p_t\\
\hat q_(e_{t+1})\\
\vdots\\
\hat q_(e_{d})
\end{pmatrix}.\]

We first consider the case when
$S=\{\hat p_0-\hat p_1,\hat p_0-\hat p_2,\ldots,\hat p_0-\hat p_t,\hat q(e_{t+1}),\ldots,\hat q(e_d)\}$ is
linearly independent. Then the rows of $A$ are linearly independent  and $\rank R(H,\hat p,\hat q)=\rank R(G,p,q)+\rank A=\rank R(G,p,q)+d$. Hence $(H,\hat p,\hat q)$ is infinitesimally rigid if and only if $(G,p,q)$ is infinitesimally rigid.

It remains to consider the case when $S$ is not
linearly independent. In this case $\dim \langle S\rangle\leq d-1$ and we may define an infinitesimal motion $\dot p$ of $(H,\hat p,\hat q)$ by choosing $\dot p(v_0)$ to be any nonzero vector in $\langle S\rangle^\perp$ and putting $\dot p(v)=0$ for all $v\neq v_0$. Hence $(H,\hat p,\hat q)$ is not infinitesimally rigid.
\end{proof}

Given a graph $G=(V,E,L)$ and $q:L\to \R^d$, let $W_v=\langle q(e)\,:\,e\in L(v)\rangle$ for all $v\in V$. We say that $q$ is {\em line-admissible on $G$} if  $\dim W_v\geq d-1$ for all $v\in V$, and $\dim \langle q(L)\rangle=d$.
A {\em cycle of length $k$}  is a connected graph with $k$ edges in which each vertex has degree two.

\begin{lem}\label{lem:cyc} Let $d\geq 2$ be an integer, $C$ be a cycle which is either a loop or has length at least three, $C^\dlm=(V,E,L)$ and $q:L\to \R^d$. Then $(C^\dlm,p,q)$ is infinitesimally rigid for some $p:V\to \R^d$ if and only if $q$ is line-admissible on $C^\dlm$.
\end{lem}

\begin{proof}
We first prove necessity.
Suppose that $q$ is not line-admissible on $C^\dlm$. If $V=\{v\}$ then $\dim W_v\leq d-1$ and any nonzero vector $\dot p_v\in W_v^\perp$ will be an infinitesimal motion of $(C^\dlm,p,q)$ for all $p$. Hence we may assume that $|V|\geq 3$. If $\dim W_v\leq d-2$ for some $v\in V$ then the rows of $R(C^\dlm,p,q)$ indexed by $L(v)$ will be dependent and hence the rank of $R(C^\dlm,p,q)$ will be less than $|E|+|L|=d|V|$ for all $p$.
Hence we may assume that $\dim W_v= d-1$ for all $v\in V$ and that $\dim \langle q(L)\rangle=d-1$. Then $W_u= W_v$ for all $u,v\in V$.
Let $W_v=W$, $t$ be a nonzero vector in $W^\perp$, and $\dot p:V\to \R^d$ be defined by $\dot p(v)=t$ for all $v\in V$. Then $\dot p$ will be a nontrivial infinitesimal motion of $(C^\dlm,p,q)$ for all $p$.

We next prove sufficiency. Suppose that $q$ is line-admissible on $C^\dlm$. If $V=\{v\}$ then $\dim W_v= d=\rank R(C^\dlm,p,q)$ for all $p$. Hence we may assume that $|V|\geq 3$. Since $\dim \langle q(L)\rangle=d$, we may choose $u,v,w\in V$ such that $uv,uw\in E$ and $W_u\neq W_v$. Let $G=(V-u,E',L')$ be the graph obtained from $C^\dlm-u$ by adding a loop $e_0$ at $v$. Define $q':L'\to \R^d$ by putting $q'(e)=q(e)$ for all $e\in L\cap L'$ and $q'(e_0)=q(e)$ for some $e\in L(u)$ with $q(e)\not\in W_v$. Then $\dim W'_v=d$ so the subgraph $H$ of $G$ induced by $v$ has an infinitesimally rigid realisation $(H,p'',q'|_H)$ for any $p''(v)\in \R^d$. We can now use Lemma \ref{lem:02ext} recursively to construct $p':V-u\to \R^d$ such that $(G,p',q')$ is infinitesimally rigid. If necessary we may perturb this realisation slightly so that $p'(v)+q'(e_0))-p'(w)\not\in W_u$.  Finally, we construct
an infinitesimally rigid realisation $(C^\dlm,p,q)$ from $(G,p',q')$ by using the proof technique of Lemma \ref{lem:extension}. More precisely we construct $C^\dlm$ from $G$ by performing a $d$-dimensional $1$-loop extension operation at the loop $e_0$ and choose $p:V\to \R^d$ such that $p|_{V-u}=p'$ and $p(u)=p'(v)+q'(e_0)$.
\end{proof}

\begin{thm}\label{thm:lingen} Let  $d\geq 2$ be an integer, $G$ be a graph, $G^\dlm=(V,E,L)$ and $q:L\to \R^d$ such that $\dim W_v\geq d-1$ for all $v\in V$. Then $(G^\dlm,p,q)$ is infinitesimally rigid for some $p:V\to \R^d$ if and only if every connected component of $G$ has a cycle $C$ of length not equal to two such that $q$ is line-admissible on $C^\dlm$.
\end{thm}
\begin{proof}
We first prove necessity.
Suppose that $G^{[d-1]}$ can be realised as an infinitesimally rigid line-constrained framework $(G^\dlm,p,q)$ in $\R^d$. Then $\rank R(G^\dlm,p,q)=d|V|$. Since  $\dim W_v\geq d-1$ for all $v\in V$, we can choose a  subgraph $H=(V,E',L')$ of $G$ such that the rows of $R(H^\dlm,p,q|_{H^\dlm})$ are linearly independent and $\rank R(H^\dlm,p,q|_{H^\dlm})=d|V|$. Then $H$ is a looped simple graph of minimum degree at least one and
$|E'|+|L'|=|V|$. If $H$ has a vertex $v$ of degree one, then we can apply Lemma \ref{lem:02ext} to
$(H^\dlm,p,q|_{H^\dlm})$ to deduce that $((H-v)^\dlm,p|_{V-v},q|_{(H-v)^\dlm})$ is infinitesimally rigid. We may then apply induction to $H-v$ to deduce that each component of $H$ has a cycle $C$ of length not equal to two such that $q$ is line-admissible on $C^\dlm$. It remains to consider the case when every vertex of $H$ has degree two. Then each component of $H$ is a cycle (which necessarily has length not equal to two since $H$ is looped simple). In this case we may use Lemma \ref{lem:cyc} to deduce that  $q$ is line-admissible on $C^\dlm$ for each component $C$ of $H$. Since every connected component of $H$ is contained in a component of $G$, each component of $G$ has a cycle $C$ of length not equal to two such that $q$ is line-admissible on $C^\dlm$.

We next prove sufficiency. Suppose that each connected component of $G$ contains a cycle $C$ of length not equal to two such that $q$ is line-admissible on $C^\dlm$. We may assume inductively that $G$ is connected and that $C$ is the unique cycle in $G$. We may now use Lemma \ref{lem:02ext} to reduce to the case when $G=C$. Lemma \ref{lem:cyc} now implies that $(G^\dlm,p,q)$ is infinitesimally rigid for some $p$.
\end{proof}

 Note that when $d=1$, a graph $G=(V,E,L)$ with a given map $q:L\to \R$ has an infinitesimally rigid realisation $(G,p,q)$ in $\R$ if and only if every connected component of $G$ has a loop $e$ with $q(e)\neq 0$. Thus Lemma \ref{lem:cyc} and Theorem \ref{thm:lingen} remain true when $d=1$ if we adopt the convention that $\dim \emptyset=0$.

\section{Body-bar frameworks}
Suppose we are given a graph $G=(V,E,L)$. We can realise $G$ as a {\em linearly-point-constrained body-bar framework in $\R^d$} by representing each vertex $v$ by a $d$-dimensional rigid body $B_v$ in $\R^d$, each edge $uv$ as a distance constraint between two points in $B_u$ and $B_v$, respectively, and each loop $e=vv$ as a constraint which restricts the motion of some point of $B_v$ to be orthogonal to a given vector $q(e)\in \R^d$.
A realisation of $G$ as a {\em linearly-body-constrained body-bar framework} is defined similarly but in this case each loop $e=vv$ represents a constraint which restricts the motion of the whole body $B_v$ to be orthogonal to a given vector $q(e)\in \R^d$.

Katoh and Tanigawa \cite{KatTan} consider closely related realisations of $G$ as a body-bar framework with either bar or pin boundary in $\R^d$. In the first case each loop $e=vv$ constrains the infinitesimal motion of a point of the body $B_v$ to be a rotation about a fixed point $P_e$. In the second case the loop constrains the infinitesimal motion of the whole body $B_v$ to be a rotation about $P_e$. They characterise when $G$ can be realised as an infinitesimally rigid body-bar framework with either a bar or pin boundary for any given set of boundary bar-directions or  pins, \cite[Corollary 6.2, Theorem 6.3]{KatTan}. Their approach uses projective geometry and considers an embedding of $\R^d$ into projective space $\bP^d$. By choosing the projective points corresponding to each point $P_e$ to lie on the hyperplane at infinity in $\bP^d$ in the case of a pin boundary, their results immediately imply the following characterisations of graphs which can be realised as infinitesimally rigid  linearly-constrained body-bar frameworks.

\begin{thm}
Let $G=(V,E,L)$ be a graph and $q:L\to \R^d\sm \{0\}$. Then $(G,q)$ can be realised as an
infinitesimally rigid  linearly-point-constrained body-bar framework in
$\R^d$ if and only if
$$\delta_G(\P) \geq  {{d+1}\choose{2}}|\P|-\sum_{X\in \P}\dim\langle q(e)^*\,:\,e\in L(X)\rangle$$
for all partitions $\P$ of $V$, where $\delta_G(\P)$ is the number of edges in $E$ joining different sets in $\P$, $L(X)$ is the set of all loops in $L$ incident to vertices of $X$, and $q(e)^*\in \R^{{}^{{d+1}\choose{2}}}$ is the vector of Pl\"ucker coordinates of a line segment in $\R^d$ with direction $q(e)$.
\end{thm}

\begin{thm} Let $G=(V,E,L)$ be a graph and $q:L\to \R^d\sm \{0\}$. Then $(G,q)$ can be realised as an infinitesimally rigid linearly-body-constrained body-bar framework in $\R^d$ if and only if
$$\delta_G(\P) \geq  {{d+1}\choose{2}}|\P| -\sum_{X\in \P, L(X)\neq \emptyset}\sum_{i=1}^{d_X+1}(d-i+1)$$
for all partitions $\P$ of $V$, where
$d_X$ is the dimension of the affine
span of $\{q(e)\,:\,e\in L(X)\}$.
\end{thm}


\section{Concluding remarks}
\label{sec:further}

\noindent 1. Theorem \ref{thm:main} gives rise to an efficient algorithm for testing whether a graph can be realised as an infinitesimally rigid framework on a type 0, $t$-dimensional variety in $\R^d$ when
$d\geq \max\{2t,t(t-1)\}$.
Details may be found in \cite{B&J,L&S}.
As discussed in the introduction, the bound $d\geq 2t$ is tight but the bound $d\geq t(t-1)$ may be far from best possible.\\

\noindent 2. The proof of Theorem \ref{thm:ST} in \cite{ST} is a direct proof using a related `frame matroid'. We briefly describe how our inductive techniques can be adapted to give an alternative proof of their result. Suppose $G=(V,E,L)$ is a  graph which satisfies the hypotheses of Theorem \ref{thm:ST} i.e.\  $|E|+|L|=2|V|$,  $|F|\leq 2|V_F|$ for all $F\subseteq E\cup L$, and, when $\emptyset\neq F\subseteq E$, $|F|\leq 2|V_F|-3$. Then $G$ contains a vertex $v$ which is incident to either 2 or 3 edges and loops. In the first case, it is easy to see that $G-v=(V',E',L')$ also satisfies the hypotheses of Theorem \ref{thm:ST}.
Hence we may use induction and Lemma \ref{lem:02ext} to deduce the rigidity of $G$.
 In the latter case, either $v$ is incident with at least one loop and we can use Lemmas \ref{lem:reduction} and \ref{lem:extension} to show that $G$ is rigid, or $v$ has three incident edges. To reduce such a vertex we may use the
 fact that the function $f:2^{E\cup L}\to \bZ$ given by $f(F)=2|V_F|-3$ if $F\subseteq E$ and $f(F)=2|V_F|$ if $F\not \subseteq E$, is nonnegative on nonempty sets, nondecreasing and crossing submodular, and hence induces a matroid on $E\cup L$, see for example \cite{JN,ST}. Let $r(G)$ denote the rank of this matroid and suppose that $N(v)=\{x,y,z\}$. Let $G'$ be formed from $G$ by deleting $v$ (and its incident edges) and adding the edges $xy,xz,yz$. Suppose that
 $r(G')=r(G)-3$. Let $G''$ be formed from $G$ by adding $xy,xz,yz$. Then $r(G'')\leq r(G')+2=r(G)-1$ since the edge set of $K_4$ is dependent in the matroid. This is a contradiction since $G$ is a subgraph of $G''$. Hence $G-v+e$ satisfies the hypotheses of Theorem \ref{thm:ST} for some $e\in \{xy,xz,yz\}$. We can now  complete the proof by choosing a generic infinitesimally rigid realisation of $G-v+e$ and placing $v$ on the line joining the end-points of $e$.\\

\noindent 3. Theorem \ref{thm:ST} was extended by Katoh and Tanigawa  \cite[Theorem 7.6]{KatTan} to allow specified directions for the linear constraints. More precisely they determine when a given  graph $G=(V,E,L)$ and map $q:L\to \R^2$ can be realised as an infinitesimally rigid linearly-constrained framework $(G,p,q)$ in $\R^2$. It is an open problem to decide if this result can be extended to plane-constrained frameworks in $\R^d$ for $d\geq 3$. In particular we offer

\begin{con} Let $G$ be a graph and $S$ be a fixed type $0$ smooth surface in $\R^d$. Then $G$ has an infinitesimally rigid realisation on $S$ if and only $G$ has a  $(2,0)$-tight looped simple spanning subgraph, which contains no copy of $K_5$ when $d=3$.
\end{con}

\end{document}